\documentclass{article}

\usepackage{amsthm}
\usepackage{amsmath}
\usepackage{amssymb}
\usepackage{tikz,tkz-graph}
\usepackage{enumitem}

\newtheorem{theorem}{Theorem}[section]
\newtheorem{conjecture}[theorem]{Conjecture}

\newtheorem{lemma}[theorem]{Lemma}

\theoremstyle{definition}

\newtheorem{problem}[theorem]{Problem}

\title{Where have all the grasshoppers gone?}
\author{J\'anos Pach \thanks{R\'enyi Institute, Budapest and IST Austria. Research partially supported by National Research, Development and Innovation Office (NKFIH) grant K-131529 and ERC Advanced Grant ``GeoScape.'' Email: pach@cims.nyu.edu.} \and G\'abor Tardos \thanks{R\'enyi Institute, Budapest. Research partially supported by National Research, Development and Innovation Office (NKFIH) grants K-132696, SNN-135643 and ERC Advanced Grants `GeoSpace'' and ``ERMiD'' Email: tardos@renyi.hu. }}
\date{}
\begin{document}
\maketitle
\begin{abstract}
Let $P$ be an $N$-element point set in the plane. Consider $N$ (pointlike) grasshoppers sitting at different points of $P$. In a ``legal'' move, any one of them can jump over another, and land on its other side at exactly the same distance. After a finite number of legal moves, can the grasshoppers end up at a point set, similar to, but larger than $P$? We present a linear algebraic approach to answer this question. In particular, we solve a problem of Brunck by showing that the answer is yes if $P$ is the vertex set of a regular $N$-gon and $N\neq 3, 4, 6$. Some generalizations are also considered.
\end{abstract}

\section{Introduction.}

In 1994, at the Moscow Mathe\-matics Olym\-piad~\cite{MO}, the following question was asked: \emph{Four grasshoppers are sitting at the vertices of a square. Every minute, one of them jumps over another and lands at the point symmetric to it. Prove that the grasshoppers can never end up sitting at the vertices of a larger square.} The problem was suggested by A. K. Kovaldzhi.
\smallskip

The solution is simple. 
Suppose, for contradiction, that the grasshoppers can end up at the vertices of a \emph{larger} square. Since the moves are reversible, starting with the final position and reversing the sequence of moves, we can get from a square to a \emph{smaller} square. However, this is impossible. Indeed, we can assume by scaling that the initial positions of the grasshoppers are $(0,0), (0,1), (1,0),$ and $(1,1)$. Then they will keep jumping around on the points of the integer grid ${\mathbb{Z}}^2$, which has no four points that form a square of sidelength smaller than 1.
\smallskip

The same argument works when we have \emph{six} grasshoppers sitting at the vertices of a regular hexagon $T$. The only difference is that now the possible positions of the grasshoppers belong to a triangular lattice induced by two adjacent side vectors of $T$. As $T$ is the smallest regular hexagon in this lattice, the grasshoppers can never reach a regular hexagonal position of smaller---and, therefore, by reversibility---of larger side length than $T$. Exactly the same argument works if we start with a regular triangle (or with an arbitrary triangle). However, in this case, we have an even simpler argument: a legal move never changes the \emph{area} of the triangle determined by the three pieces. Therefore, after any number of steps, the positions of the pieces must form a triangle whose area is the same as that of the initial triangle.
\smallskip

As far as we know, it was Florestan Brunck~\cite{B} who first asked: what happens if there are $5$ grasshoppers, and their starting position is the vertex set of a regular pentagon? Can they be taken into the vertices of a larger regular pentagon?

Somewhat surprisingly, the answer to this question is affirmative.

\begin{theorem}\label{main}
Let $N\ge 3$, and let $C_N$ denote a configuration consisting of $N$ pieces sitting at the vertices of a regular $N$-gon. In a legal move, any piece can jump over any other piece and land at the point centrally symmetric about it.

Unless $N=3, 4,$ or $6$, there is a finite number of moves that take the pieces of $C_N$ into a strictly larger similar configuration.
\end{theorem}

One can raise the following more general question. Given any two $N$-element point sets in the plane or in higher dimensions, is it possible to transform the first one to the second by a sequence of legal moves? In Section~\ref{sec3}, we develop a linear algebraic approach to answer this question.
\smallskip

Throughout this note, we fix an orthogonal system of coordinates in ${\mathbf{R}}^d$. Let $P\subset {\mathbf{R}}^d$ be a configuration of $N=n+1$ points.
The position of each piece can be described by a \emph{column vector} of length $d$. Assume without loss of generality that initially one of the pieces is at the origin $0$. We distinguish this piece and call it \emph{stationary}. All other pieces are said to be \emph{ordinary}. The movement of the stationary piece is artificially restricted: any ordinary piece is allowed to jump over it, but this piece is not allowed to jump. Note that it is not forbidden for two or more pieces to sit at the same point at the same time.
\smallskip

Enumerate the elements of $P$ by the integers $0, 1, 2,\ldots, n$, where $0$ denotes the stationary piece. We identify $P$ with the $d\times n$ real matrix, whose $i$\/th column gives the position of the ordinary piece $i$ for $1\le i\le n$. Without any danger of confusion, we denote both the configuration and the corresponding matrix by $P$. Part 1 of Lemma~\ref{altalanos}, stated in the next section, characterizes all configurations $P$ that can be reached from the initial position under the restriction that piece $0$ is stationary. In part 2, we get rid of this unnatural condition.
\smallskip

In the special case where we want to transform $P$ into a larger configuration similar to it, we obtain the following result.

\begin{theorem}\label{hasonlo}
Let $P$ be a configuration of $N=n+1$ points in ${\mathbf{R}}^d$, one of which initially sits at the origin but it is still allowed to jump. Let $P$ also denote the $d\times n$ matrix associated with it.

Then $P$ can be transformed into a similar but larger configuration, using legal moves, if and only if there is an $n\times n$ integer matrix $A$ with $|\det A|=1$ such that the configuration associated with $PA$ is similar to and larger than the initial configuration $P$.
\end{theorem}

Our paper is organized as follows. In Section~\ref{sec2}, we prove a simple lemma for matrices. It is instrumental for our general linear algebraic approach to the reconfiguration problem, presented in Section~\ref{sec3}. In Section~\ref{sec4new}, we prove our main theorems. First we establish Theorem~\ref{hasonlo}, and then we show how to deduce from it Theorem~\ref{main}. In the last section, we mention some open problems and make a few remarks.
\smallskip

The problem discussed in our paper is broadly related to several topics in pure and applied mathematics and computer science. For instance, the theory of \emph{discrete dynamical systems} deals with models in which changes occur in discrete time steps, according to a fixed rule; see~\cite{Ga}. A \emph{cellular automaton} consists of an array of cells, each of which can be in a finite number of different states. The states get updated depending on the neighborhoods of the cells~\cite{Sch, Wo}. Perhaps the most famous cellular automaton is \emph{Conway's game of life}~\cite{JoG}. These models are mostly deterministic. Solitary games in which the player can choose one from several possible legal moves, are studied in the framework of \emph{combinatorial game theory}~\cite{AlNW, BeCG}. Similar models are analyzed in \emph{motion planning} and \emph{robotics}, related to the reconfiguration of rotating or sliding systems of modules~\cite{DuSY}.

\section{Admissible transformations of matrices.}\label{sec2}

We follow the notation introduced above: A configuration $P$ consisting of $n$ ordinary pieces and a stationary piece at the origin in ${\mathbf{R}}^d$ is identified with a $d\times n$ matrix, also denoted by $P$, the columns of which correspond to the positions of the ordinary pieces.
\smallskip

The legal move by which the ordinary piece $i$ ($1\le i\le n$) jumps over piece $j$ ($0\le j\le n$ and $i\ne j$) brings configuration $P$ to $PA_{ij}$, where $A_{ij}=(a^{(ij)}_{kl})_{k,l=1}^n$ is the $n\times n$ real matrix defined as follows.
\[a^{(ij)}_{kl}=\left\{\begin{array}{ll}
1&\text{for }k=l\ne i\\
-1&\text{for }k=l=i\\
2&\text{for }k=j,l=i\\
0&\text{otherwise.}
\end{array}\right.\]
For example, for $n=4$, we have
\[A_{2,4}=\begin{bmatrix}
1&0&0&0\\
0&-1&0&0\\
0&0&1&0\\
0&2&0&1
\end{bmatrix}\quad\quad  \mbox{and} \quad\quad
A_{2,0}=\begin{bmatrix}
1&0&0&0\\
0&-1&0&0\\
0&0&1&0\\
0&0&0&1
\end{bmatrix},\]
where the first matrix corresponds to the jump of piece $2$ over piece $4$, and the second matrix to the jump of $2$ over the stationary piece $0$. Indeed, denoting the position vectors of $2$ and $4$ by $v_2$ and $v_4\in{\mathbf{R}}^d$, respectively, after the first move, the new position of $2$ will be $v_4 + (v_4-v_2)=-v_2+2v_4$. This change of coordinates is reflected by the second column of $A_{2,4}$.
\smallskip

We call the matrices $A_{ij}$ for $1\le i\le n$ and $0\le j\le n$, $i\ne j$, \emph{elementary involutions}. (Note that they are indeed involutions, that is, $A_{ij}^2=I$ holds, reflecting the fact that all legal moves are reversible. Here $I$ stands for the identity matrix.) The positions reachable from an initial position $P$ (with the piece at the origin remaining stationary) are precisely the configurations $PA$, where $A$ belongs to the \emph{matrix group} generated by the \emph{elementary involutions}.
\smallskip

Our first lemma  gives a simple characterization of this matrix group. It will be used in the next section. We write $\det A$ for the determinant of the square matrix $A$.

\begin{lemma}\label{matrix}
An $n\times n$ integer matrix $A=(a_{ij})$ is generated by elementary involutions if and only if
\begin{enumerate}
\item[(i)] $|\det A|=1$;
\item[(ii)] $a_{ij}$ is \emph{even} if $i\neq j$ and it is \emph{odd} if $i=j$.
\end{enumerate}
\end{lemma}

\begin{proof}
All elementary involutions have determinant $-1$ and all are equal to the identity matrix $I$ modulo $2$. Conditions (i) and (ii) are closed under taking products. That is, if $A$ and $B$ meet them, then so does $AB$. This proves the ``only if'' part of the lemma.
\smallskip

It is instructive to think of the elementary involutions as column operations: For $1\le i\le n$ multiplying by $A_{i0}$ from the right is the same as multiplying column $i$ by $-1$. If we further have  $1\le j\le n$ and $j\ne i$, then multiplying by $A_{ij}A_{i0}$ is the same as \emph{subtracting} twice column $j$ from column $i$, while multiplying by $A_{i0}A_{ij}$ is the same as \emph{adding} column $j$ twice to column $i$. We call these column operations \emph{admissible}.

We prove the ``if'' part of the lemma by showing that any matrix satisfying conditions~(i) and (ii) in the lemma can be transformed into the identity matrix $I$ by admissible operations. We do this by induction on $n$. For $n=1$, the statement is trivially true. Suppose now that $n>1$ and that our statement holds for $n-1$. Let $A$ be an $n\times n$ matrix satisfying conditions (i) and (ii).
\smallskip

Consider the last row of $A$. As long as we can find two entries $a_{ni}$ and $a_{nj}$ with $|a_{ni}|>|a_{nj}|>0$, by subtracting or adding twice the $j$th column to the $i$th column, we can reduce the value of $|a_{ni}|$. In this step, no other entry of the last row changes. Hence, the sum of the absolute values of the elements of the last row strictly decreases. After a finite number of such steps (admissible operations), we will get stuck. At this point, every entry of the last row is either $0$ or $\pm a$ for some integer $a>0$. During this process, the absolute value of the determinant of the matrix has remained $1$. Now it is divisible by $a$, so we must have $a=1$. Since the parities of the elements have not changed, all entries of the last row must be $0$, with the exception of the last element, which is either $+1$ or $-1$. Applying a single admissible transformation if necessary, we can assume that this last element is $+1$.
\smallskip

Let $B$ denote the matrix constructed so far, and let $B'$ be the $(n-1)\times(n-1)$ submatrix of $B$, obtained by deleting its last column and last row. Expanding the determinant of $B$ along its last row, we get that $|\det B|=|\det B'|=1$. Clearly, $B'$ also satisfies condition (ii) of the lemma. Thus, we can apply the induction hypothesis to $B'$, and conclude that $B'$ can be transformed to an $(n-1)\times(n-1)$ identity matrix by performing a number of admissible transformations. These transformations can, of course, be extended to the whole matrix $B$. During this process, the last row of $B$ remains the same: all of its entries, except the last one, remain $0$. Its last column also remains unchanged.
\smallskip

The matrix obtained so far coincides with the identity matrix $I$, apart from the entries of its last column that are arbitrary even integers, except for the last value of the last column, which is equal to $1$.. We can make them vanish by adding or subtracting each of the first $n-1$ columns by an even number of times. Thus, by a sequence of admissible transformations, $A$ can be transformed into $I$, as required.
\end{proof}

\section{Char\-ac\-ter\-i\-za\-tion of all reach\-able po\-si\-tions.}\label{sec3}
Our next state\-ment gives a complete linear algebraic description of all configurations that can be obtained from a given initial position of the pieces by a sequence of legal moves.

\begin{lemma}\label{altalanos}
Let $P$ be a $d\times n$ matrix associated with the initial configuration of the $n+1$ pieces in ${\mathbf{R}}^d$, one of them in the origin.
\begin{enumerate}
\item
The configurations that can be reached by a sequence of legal moves that keep the stationary piece at the origin, are exactly the positions corresponding to a matrix $PA$, where $A$ is a matrix generated by the elementary involutions.
\item
If we drop the assumption that the stationary piece must stay at the origin and allow it to jump over any of the remaining $n$ pieces, every reachable configuration is a translate of some configuration that can be reached without moving the stationary piece. Moreover, a translate of such a configuration is reachable if and only if the translation vector is of the form $2Pw$ for an integer vector $w$.
\end{enumerate}
\end{lemma}

\begin{proof}
Part~1 is immediate from the observation made in the previous section, according to which the effect of a legal jump by an ordinary piece on the matrix associated with the configuration is exactly a multiplication by an elementary involution.
\smallskip

For the proof of part~2, call the piece originally sitting at the origin \emph{special}. (We can no longer call it stationary.) All other pieces are ordinary. Let us define for any ordinary piece $A$ a sequence $S_A$ consisting of $2n$ legal jumps as follows. Let $S_A$ start with the moves in which all other pieces (including the special piece) jump over $A$, and continue with the moves where all ordinary pieces jumps over the special piece. Notice that if we apply the sequence $S_A$ to any configuration, all pieces will move with the same vector: twice the vector from the special piece to $A$ (either before or after the execution of $S_A$). Applying the same jumps in reverse order, we can move all pieces with the negative of this vector. This means that applying several of these sequences right at the start, we can move all the pieces with $2Pw$, for any integer vector $w$. We can thereafter keep the special piece stationary and reach the translate of any configuration that was reachable in part~1. This proves that all configurations described in part~2 of the theorem are indeed reachable by a sequence of legal moves.
\smallskip

Next we show that any configuration reachable by legal moves is a translate of a configuration reached by legal moves keeping the special piece stationary. Consider an arbitrary sequence of legal moves reaching a configuration $X$. For every ordinary piece $A$, let us insert the sequence $S_A$ of jumps after every time the special piece jumps over $A$. Clearly, these inserted sub-sequences just act as translations, so we obtain a longer sequence of legal moves resulting in a translate of $X$. However, as $S_A$ starts with the special piece jumping over $A$, this longer sequence will have all jumps of the special piece reversed right away, so the same configuration can also be reached canceling these jumps and keeping the special piece stationary, as claimed.
\smallskip

It remains to prove that the translation vector, i.e., the position of the special piece in a configuration reached by legal moves, is of the form $2Pw$ for an integer vector $w$. For this, notice that the pieces remain in the additive subgroup of the space generated by the original positions of the pieces. This subgroup is $L=\{Pw\mid w\hbox{ integer}\}$. It may be a discrete lattice but may even be dense in the entire space. Whenever any piece jumps, it moves with twice a vector from $L$, so the special piece, which starts at the origin, must end up in $2L=\{2Pw\mid w\hbox{ integer}\}$, as claimed.
\end{proof}

\section{Proof of main theorems.}\label{sec4new}

Theorem~\ref{hasonlo} provides a simpler condition than Lemma~\ref{altalanos} for the special case where the final configuration is a similar but larger copy of the initial one. A configuration is similar to but larger than the one identified with the matrix $P$ if and only if its matrix is $\lambda SP$, where $|\lambda|>1$ and $S$ is a $d\times d$ \emph{orthogonal} matrix.

\medskip

\textbf{Proof of Theorem~\ref{hasonlo}:}
The proof of this result would be an immediate consequence of Lemma~\ref{altalanos} and Lemma~\ref{matrix} if we had also required that the matrix $A$ satisfies the parity condition, i.e., $A$ is congruent to the identity matrix $I$ modulo 2. The main observation here is that this requirement can be dropped.
\smallskip

To see this, assume that $A$ satisfies the conditions of Theorem~\ref{hasonlo}. We have $|\det A|=1$, but as the parity condition is dropped, $A$ is not necessarily generated by elementary involutions.
\smallskip

Consider the matrix $A$ modulo 2. Its determinant is 1, so it is an element of the $n\times n$ \emph{matrix group} $GL(n,2)$. As $GL(n,2)$ is a \emph{finite} group, $A$ must have a finite order $t\ge1$, for which $A^t$ agrees with the identity matrix modulo 2. We also have $|\det(A^t)|=|\det A|^t=1$. Hence, $A^t$ satisfies both conditions in Lemma~\ref{matrix} and is, therefore, generated by elementary involutions. Hence, by Lemma~\ref{altalanos}, $PA^t$ is reachable from the original configuration $P$ by legal jumps (leaving the special piece stationary). By our assumption on $A$, we have $PA=\lambda SP$, for some $|\lambda|>1$ and some orthogonal matrix $S$. This implies that $PA^t=\lambda^tS^tP$, which means that the configuration $PA^t$ is similar to but larger than $P$. This completes the proof.
\hfill $\Box$
\medskip

Now we return to the original question addressed in our paper: to the case of regular polygons.
\medskip

\textbf{Proof of Theorem~\ref{main}:}
Let $N=n+1$ and let $p_0, p_1, p_2,\ldots,p_{n}\in{\mathbf{R}}^2$ denote the vertices of a regular $N$-gon in the plane, enumerated in positive (counter-clockwise) cyclic order, and suppose that $p_0=0$. The initial position with the $N$ pieces at these vertices is identified with the $2\times n$ matrix $P$, the $i$\/th column of which is $p_i$, for all $1\le i\le n$.
\bigskip

\begin{tikzpicture}[scale=.03]
\hskip2.5cm
	\node[circle, fill=black](p0) at (-59,-81) {};
	\node[circle, fill=black](p1) at (59,-81) {};
	\node[circle, fill=black,label={$p_2$}](p2) at (95,31) {};
	\node[circle, fill=black,label={$p_3$}](p3) at (0,100) {};
	\node[circle, fill=black,label={$p_4$}](p4) at (-95,31) {};
	\draw[-latex,color=black,thick] (p0) -- (p2) node[midway, above left]{$p_2$};
	\draw[-latex,color=black,thick] (p1) -- (p3) node[midway, above left]{$p_3-p_1$};
	\draw[-latex,color=black,thick] (61,7) arc[radius=30,start angle=36, end angle =108];
	\node at (42,5) {$\frac{2\pi}5$};
	\node at (-65,-65) {$0=p_0$};
	\node at (66,-65) {$p_1$};
	
\end{tikzpicture}
\bigskip
\medskip

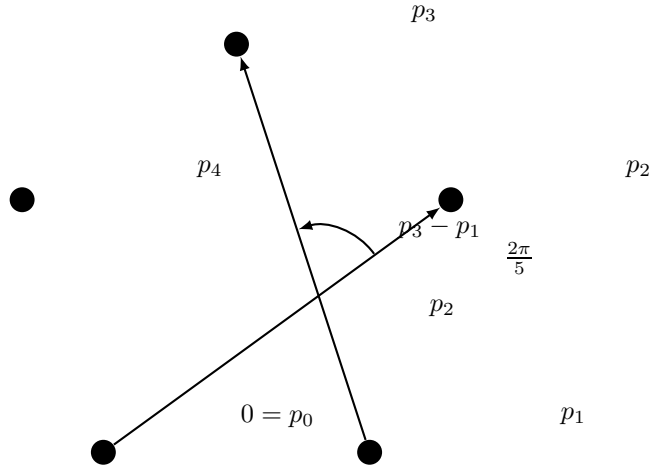
\centerline{\small{\textbf{Figure 1:} Illustration to the proof of Theorem~\ref{main}.}}

\bigskip
\medskip

Let $M$ denote the $n\times n$ matrix
$$M = \begin{bmatrix}
    -1 & -1 &  \cdots & -1 & -1 \\
    1 & 0 & \cdots &  0 & 0 \\
    0 & 1 & \cdots & 0 & 0 \\
    \vdots & \vdots &  \ddots & \vdots & \vdots \\
    0 & 0 &  \cdots  & 1 & 0
    \end{bmatrix}.$$
Note that the $i$\/th column of $PM$ is $p_{i+1}-p_1$ for $1\le i<n$ and $-p_1$ for $i=n$ and thus can be obtained as the rotation of $p_i$ with the positive angle $2\pi/N$ (see Figure 1). Let
$$S=\begin{bmatrix}
\cos\frac{2\pi}{n+1}&-\sin\frac{2\pi}{n+1}\\
&\\
\sin\frac{2\pi}{n+1}&\cos\frac{2\pi}{n+1}
\end{bmatrix}$$
be the orthogonal matrix representing this rotation. Now we have $PM=SP$.
\smallskip

At this point, we are almost done. We have $|\det M|=1$, and the position $PM$ is similar to $P$, that is, it is also a regular $(n+1)$-gon. If it was also larger than $P$, we could apply Theorem~\ref{hasonlo} and finish the proof. Unfortunately, the configuration $PM=SP$ is of exactly the same size as $P$, because $S$ is a rotation.
\smallskip

We solve this problem by considering polynomials $f(M)$ of $M$. Obviously, we have $Pf(M)=f(S)P$. Notice that all polynomials $f(S)$ of the matrix $S$ are of the form
$$\begin{bmatrix}
a&-b\\
b&a
\end{bmatrix}=\lambda T,$$
for some reals $a,b,\lambda$ and for an orthogonal matrix $T$. So, the configuration $Pf(M)$ is similar to $P$, i.e., it is a (possibly degenerate) regular $(n+1)$-gon. We can finish the proof by applying Theorem~\ref{hasonlo} if we find an integer polynomial $f$ such that
\begin{itemize}
\item[(i)] $|\det f(M)|=1$, and
\item[(ii)] the regular $(n+1)$-gon identified with $Pf(M)$ is larger than $P$.
\end{itemize}

For a positive integer $k$, consider the matrix
$$B_k=\sum_{j=0}^{k-1}M^j.$$
Note that the characteristic polynomial of $M$ is $\sum_{j=0}^{n}t^j$. Recall that $N=n+1$. Thus, we have $B_N=0$. We further have
$$B_{kN}=B_N\sum_{j=0}^{k-1}M^{jN}=0,$$
for all $k>0$, and
$$B_{kN+1}=I+MB_{kN}=I.$$
\smallskip

Let us fix an integer $i$ with $1 \le i<N$ which is coprime to $N$. We can find positive integers $l$ and $k$ such that $il=kN+1$, and
$$B_i\sum_{j=0}^{l-1}M^{ij}=B_{il}=B_{kN+1}=I.$$
This shows that the integer matrix $B_i$ is invertible and its inverse is also an integer matrix. Therefore, we have $|\det B_i|=1$. This was condition (i) above that a polynomial of $M$ had to satisfy to finish the proof.
\smallskip

It remains to verify condition (ii) and check that the configuration identified with $PB_i$ (a regular $N$-gon) is strictly larger than the starting configuration. The side length of the $N$-gon associated with the configuration $PB_i$ is the distance between $p_1B_i$, the position of piece 1, and the special piece at the origin. Here, $p_1M^j=p_{j+1}-p_j$ for $0\le j<N-1$, so that
$$p_1B_i=\sum_{j=0}^{i-1}p_1M^j=\sum_{j=0}^{i-1}(p_{j+1}-p_j)=p_i.$$
Therefore, the side length of the regular $(n+1)$-gon corresponding to the configuration $PB_i$ is equal to $|p_i|$. This is the length of a proper diagonal of the original configuration $P$ provided that $1<i<N-1$, so strictly larger than original side length of $|p_1|$.
\smallskip

To complete the proof of Theorem~\ref{main}, it is enough to observe that there always exists an integer $i$ with  $1<i<N-1,$ which is coprime to $N$, provided that $N\ge5$ and $N\ne6$.  \hfill   $\Box$

\section{Remarks and further questions.}\label{remarks}

In the present note, we defined a game with grasshoppers jumping over each other and discussed when a configuration of the grasshoppers \emph{can be reached} from another configuration but we neglected the question of \emph{how many jumps} it takes to reach the desired configuration. Here, we list several problems related to this latter question.
\medskip

\noindent
\textbf{A. } The smallest special case of Theorem~\ref{main} is that of the regular pentagon. One can implement the steps of the proof above to find a concrete sequence of jumps for the grasshoppers starting at the vertices of a regular pentagon that yields a larger regular pentagon, but the sequence so obtained is very long. Here we give a much shorter sequence meeting the requirements. We number the grasshoppers $0$ through $4$ and denote by $J_{ij}$ the jump of grasshopper $i$ over grasshopper $j$. The following sequence of jumps results in a regular pentagon that is $\sqrt5+2$ times larger than the original:
$$J_{40}J_{43}J_{20}J_{32}J_{41}J_{12}J_{13}J_{31}J_{23}J_{13}J_{30}J_{24}J_{14}J_{40}$$

We do not know whether there exists an even shorter sequence than this.
\medskip

\noindent
\textbf{B. } The same question can be asked more generally:

\begin{problem}\label{numberofmoves}
Given an integer $N$ which is either $5$ or larger than $6$, and $N$ pieces sitting at the vertices of a regular $N$-gon, what is the smallest number of legal moves that takes them into the vertex set of a larger regular $N$-gon?

How many legal moves are needed to turn a configuration consisting of the vertices and the center of regular $N$-gon to a similar but larger configuration? (A finite number of moves always suffice.)
\end{problem}

We can turn to the proof of Theorem~\ref{main} to obtain an explicit bound, but this bound will be unreasonably large in $N$ for two reasons. First, the application of Lemma~\ref{matrix} may itself yield exponentially long sequences; see more on this below. Second, in the proof of Theorem~\ref{hasonlo} we take an integer matrix $A$ and raise it to a power which is the order of $A$ in $GL(n,2)$. But $GL(n,2)$ has elements of order $2^n-1$.

\begin{conjecture} Every regular $N$-gon ($N>6$) can be transformed to a similar but larger configuration in polynomially many moves in $N$.
\end{conjecture}

\noindent
\textbf{C. } We can also ask how many legal jumps are needed to bring a given configuration $P$ to another configuration $Q$ whenever this is possible. Here we cannot hope for a bound that depends only on the number $N$ of pieces in the configurations. Lemma~\ref{altalanos} claims that if $P$ can be transformed to $Q$ at all, then $Q=PA$ for an $(N-1)\times(N-1)$ matrix $A$ generated by the elementary involutions. (Here we assume that both $P$ and $Q$ have a stationary piece in the origin and we identify the configurations with matrices as usual.) It would be interesting to bound the number of legal moves needed in terms of $N$ and the maximum absolute value of an entry in $A$. This is the same as asking how long a product of elementary transformations may be needed to obtain a specific matrix with bounded entries.

As before, our proof of Lemma~\ref{altalanos} can be turned into an explicit bound. However, this bound will be exponential. The reason is that the algorithmic proof of Lemma~\ref{matrix}, turning a matrix into the identity matrix by a series of admissible operations, is recursive. We do not know whether this is necessary.
\medskip

\noindent
\textbf{D.} What happens if we are given a single configuration of $N$ pieces which comes with a guarantee that it can be transformed into a similar but larger configuration using legal jumps? Can we bound the number of jumps needed solely in terms of $N$?
\medskip

\noindent
\textbf{E.} And finally a question \emph{not} about the number of necessary jumps.
\begin{problem}
For which Platonic solids $T$ can we start the above game at the vertices of $T$ and arrive at a similar but larger configuration, using legal moves?
\end{problem}
If $T$ is a tetrahedron, a cube or an octahedron, then this is clearly impossible, because all vertices of these solids lie on a cubic lattice. Just like in the case of the olympiad problem mentioned at the beginning of this article, if we consider the smallest similar copy of $T$ whose vertices belong to the integer lattice of side length $1$, we cannot end up with a smaller configuration similar to $T$.  By the reversibility of the moves, neither can we arrive at a similar, but larger configuration.

We conjecture that starting from the vertices of either of the remaining two Platonic solids, the dodecahedron or the icosahedron, one can arrive to a similar but larger configuration by a series of legal moves.
\bigskip

\textbf{Acknowledgement.} We are grateful to Florestan Brunck for communicating this problem to us. The authors, together with him, posed the question on the 2022 Mikl\'os Schweitzer Competition for Hungarian undergraduates and MSc students. This is arguably the most difficult mathematics contest of the world. The participants have 10 days to solve 10 problems~\cite{Sz}.

\end{document}